\documentclass[12pt,reqno]{article}

\usepackage[usenames]{color}
\usepackage{amssymb}
\usepackage{graphicx}
\usepackage{amscd}
\usepackage{pstricks}

\usepackage[colorlinks=true,
linkcolor=webgreen,
filecolor=webbrown,
citecolor=webgreen]{hyperref}

\definecolor{webgreen}{rgb}{0,.5,0}
\definecolor{webbrown}{rgb}{.6,0,0}

\usepackage{color}
\usepackage{fullpage}
\usepackage{float}

\usepackage{amsmath,amssymb}
\usepackage{amsthm}
\usepackage{amsfonts}
\usepackage{latexsym}
\usepackage{epsf}
\usepackage{tikz}
\usepackage{tikz-qtree}
\usepackage{algorithmicx}
\usepackage{algpseudocode}
\usepackage{algorithm}

\newcommand{\seqnum}[1]{\href{http://oeis.org/#1}{\underline{#1}}}

\begin{document}


\theoremstyle{plain}
\newtheorem{theorem}{Theorem}
\newtheorem{corollary}[theorem]{Corollary}
\newtheorem{lemma}[theorem]{Lemma}
\newtheorem{proposition}[theorem]{Proposition}

\theoremstyle{definition}
\newtheorem{definition}[theorem]{Definition}
\newtheorem{example}[theorem]{Example}
\newtheorem{conjecture}[theorem]{Conjecture}
\newtheorem{question}[theorem]{Question}

\theoremstyle{remark}
\newtheorem{remark}[theorem]{Remark}

\begin{center}
\vskip 1cm{\LARGE\bf Egyptian Fractions and Prime Power Divisors}  \vskip 1cm
\large
John Machacek\\
Department of Mathematics\\
Michigan State University\\
USA\\
\href{mailto:machace5@math.msu.edu}{\tt machace5@math.msu.edu} \\
\end{center}

\vskip .2 in

\begin{abstract}
From varying Egyptian fraction equations we obtain generalizations of primary pseudoperfect numbers and Giuga numbers which we call prime power pseudoperfect numbers and prime power Giuga numbers respectively.
We show that a sequence of Murthy in the OEIS is a subsequence of the sequence of prime power pseudoperfect numbers.
We also provide prime factorization conditions sufficient to imply that a number is a prime power pseudoperfect number or a prime power Giuga number.
The conditions on prime factorizations naturally give rise to a generalization of Fermat primes which we call extended Fermat primes.
\end{abstract}

\section{Introduction}
\label{sec:Intro}
We define and study two new types of integers which we call prime power pseudoperfect numbers and prime power Giuga numbers.
Each satisfies an Egyptian fraction equation that is a variation of a previously studied Egyptian fraction equation.
Throughout we will reference relevant sequences from the OEIS~\cite{OEIS}.
In Section~\ref{sec:Intro} we review pseudoperfect numbers, primary pseudoperfect numbers, and Giuga numbers.
In Section~\ref{sec:pseudo} we will define prime power pseudoperfect numbers and show their relation to a sequence in the OEIS contributed by Murthy.
Prime power Giuga numbers are defined in Section~\ref{sec:Giuga}.
We will give some formulas that can produce more terms of our sequences in Section~\ref{sec:more}.
In Section~\ref{sec:open} we discuss some open problems and introduce extended Fermat primes.
We have contributed each of our new sequences to the OEIS.
The sequences of prime power pseudoperfect numbers, prime power Giuga numbers, and extended Fermat primes are \seqnum{A283423}, \seqnum{A286497}, and \seqnum{A286499} respectively.

A \emph{pseudoperfect number} is a positive integer $n$ such that there exist $0 < d_1 < \cdots < d_k < n$ where $d_i \mid n$ for each $i$ and $n = d_1 + \cdots + d_k$.
For example, the number $20$ is a pseudoperfect number since $20 = 1 + 4 + 5 + 10$,
Pseudoperfect numbers were first considered in the article~\cite{pseudo} and are sequence~\seqnum{A005835}.
Open problems on pseudoperfect numbers can be found in the book~\cite[B2]{Guy}.
A \emph{primary pseudoperfect number} is a positive integer $n > 1$ which satisfies the Egyptian fraction equation
\[ \sum_{p \mid n} \frac{1}{p} + \frac{1}{n} = 1\]
where the sum is taken over all prime divisors on $n$.
Primary pseudoperfect numbers were originally defined in the article~\cite{PPN} and are sequence~\seqnum{A054377}.
When $n>1$ is primary pseudoperfect number it follows that
\[ \sum_{p \mid n} \frac{n}{p} + 1 = n.\]
So, we see that, with the exception of $2$, every primary pseudoperfect number is a pseudoperfect number.

A \emph{Giuga number} is a positive composite integer $n$ such that
\[\sum_{p \mid n} \frac{1}{p} - \frac{1}{n} \in \mathbb{N} \]
where the sum is taken over all prime divisors on $n$.
Giuga numbers were defined in the article~\cite{GiugaNum} and are sequence~\seqnum{A007850}.
All known Giuga numbers satisfy the stronger Egyptian fraction equation
\[ \sum_{p \mid n} \frac{1}{p} - \frac{1}{n} = 1.\]
Giuga numbers are related to Giuga's conjecture on primality~\cite{Giuga}.
Open problems relating to Giuga numbers can be found in the book~\cite[A17]{Guy}

\section{Prime power pseudoperfect numbers}
\label{sec:pseudo}

A \emph{prime power pseudoperfect number} is a positive integer $n > 1$ which satisfies the Egyptian fraction equation
\[ \sum_{p^k \mid n} \frac{1}{p^k} + \frac{1}{n} = 1\]
where the sum is taken over all prime power divisors of $n$.
As an example we can check that $20$ is a prime power pseudoperfect number since
\[\frac{1}{2} + \frac{1}{4} + \frac{1}{5} + \frac{1}{20} = 1.\]
Observe that, with the exception of powers of $2$, all prime power pseudoperfect numbers are pseudoperfect.
Also note that any primary pseudoperfect number is a prime power pseudoperfect number since primary pseudoperfect numbers must be squarefree.
Prime power pseudoperfect numbers are sequence~\seqnum{A283423}.

We will now consider the sequences~\seqnum{A073932} and~\seqnum{A073935} both of which were contributed to the OEIS by Murthy.
We first define a function $d$ on composite numbers by letting $d(n)$ denote the largest nontrivial divisor of $n$.
For example, $d(15) = 5$.
Next we define a function $f$  on positive integers greater than $1$ by
\[ f(n) := \begin{cases}n-1, & \text{if $n$ is prime;}\\ n - d(n), & \text{otherwise.}  \end{cases} \]
As an example, $f(15) = 10$.
Given any positive integer $n > 1$ we can iterate the function $f$ until we reach $1$.
We let $f^{(i)}$ to be the $i$th iterate of the function $f$.
So, $f^{(i)}$ is the composition $f \circ f^{(i-1)}$ for $i \geq 1$ and $f^{(0)}(n) = n$.
In this way we obtain a triangle with $n$th row given by $n, f(n), f^{(2)}(n), \dots, 1$.
The sequence~\seqnum{A073932} is the sequence consisting of the entries of this triangle read by rows.
For any positive integer $n$ let $D_n$ denote the set of divisors of $n$, and we define $F_n := \{n,f(n), f^{(2)}(n), \dots, 1\}$.
The sequence consisting of all $n$ such that the $F_n = D_n$ is sequence~\seqnum{A073935}.

If we consider $n=20$ we obtain
\begin{align*}
    f(20) &= 10\\
    f^{(2)}(20) &= 5\\
    f^{(3)}(20) &= 4\\
    f^{(4)}(20) &= 2\\
    f^{(5)}(20) &= 1
\end{align*}
which are exactly the divisors of $20$.
We also notice that $20$ is a prime power pseudoperfect number since
\[ \frac{1}{2} + \frac{1}{4} + \frac{1}{5} + \frac{1}{20} = 1.\]
We will show in Theorem~\ref{thm:pseudo} that every number in the sequence~\seqnum{A073935}, with the exception of $1$, is a prime power pseudoperfect number.
We first prove two lemmata.

\begin{lemma}
Let $n > 1$ be a positive integer with prime factorization $n = p_1p_2 \cdots p_{\ell}$ where $p_1 \leq p_2 \leq \cdots \leq p_{\ell}$.
The function $f$ is then given by
\[ f(n) = (p_1 - 1)p_2 \cdots p_{\ell}.\]
\label{lem:minusone}
\end{lemma}
\begin{proof}
Take any positive integer $n > 1$ with prime factorization $n = p_1p_2 \cdots p_{\ell}$ where $p_1 \leq p_2 \leq \cdots \leq p_{\ell}$.
If $n$ is prime, then $\ell=1$ and $n = p_1$.
In this case $f(n) = n - 1 = p_1 -1$.
When $n$ is composite $\ell > 1$ the largest nontrivial divisor is $p_2p_3 \cdots p_{\ell}$.
In this case
\begin{align*}
    f(n) &= n - d(n)\\
    &= p_1p_2 \cdots p_{\ell} - p_2p_3 \cdots p_{\ell}\\
    &= (p_1 - 1)p_2\cdots p_{\ell}.
\end{align*}
We see in any case the lemma holds.
\end{proof}

\begin{lemma}
Let $n > 1$ be a positive integer with prime factorization
\[n = \prod_{i=1}^{\ell} p_i^{a_i}\]
where $p_1 < p_2 < \cdots < p_{\ell}$.
The integer $n$ satisfies $F_n = D_n$ if and only if
\[(p_i-1) = \prod_{j = 1}^{i-1}p_j^{a_j}\]
for $1 \leq i \leq \ell$.
\label{lem:A073935}
\end{lemma}

\begin{proof}
Take any positive integer $n > 1$ with prime factorization
\[n = \prod_{i=1}^{\ell} p_i^{a_i}\]
where $p_1 < p_2 < \cdots < p_{\ell}$.
The divisors of $n$ are
\[D_n = \left\{\prod_{i = 1}^{\ell} p_i^{b_i} : 0 \leq b_i \leq a_i\right\}.\]
We must show $F_n = D_n$ if and only if
\[(p_i-1) = \prod_{j = 1}^{i-1}p_j^{a_j}\]
for $1 \leq i \leq \ell$.

First assume that
\[(p_i-1) = \prod_{j = 1}^{i-1}p_j^{a_j}\]
for $1 \leq i \leq \ell$.
It follows that $F_n = D_n$ as the divisors of $n$ are obtained in lexicographic order of exponent vectors when we iterate $f$.
To see this consider $d \in D_n$ with
\[d = \prod_{i = 1}^{\ell} p_i^{b_i}\]
for some $0 \leq b_i \leq a_i$ and $d \neq 1$.
Let $i'$ be smallest index with $b_{i'} > 0$; so, then
\[d = \prod_{i = i'}^{\ell} p_i^{b_i}.\]
If $i' = 1$, then
\[f(d) = p_1^{b_1-1} \prod_{i = 2}^{\ell} p_i^{b_i}.\]
Otherwise $j > 1$ and
\[f(d) = \left(\prod_{i=1}^{i'-1} p_i^{a_i}\right)\left(p_{i'}^{b_{i'} - 1}\right)\left(\prod_{i=i'+1}^{\ell} p_i^{b_i}\right).\]
In either case we obtain the next divisor of $n$ in lexicographic order of exponent vectors.

Next assume that $F_n = D_n$.
We claim that $p_1 = 2$.
To see this, note that $(p_1 - 1) \mid f(n)$ by Lemma~\ref{lem:minusone}.
Since $F_n \subseteq D_n$ we have $f(n) \mid n$.
Thus any prime divisor of $p_1 - 1$ divides $n$, but a prime divisor of $p_1 - 1$
will be a prime strictly less than the smallest prime divisor of $n$, namely
of $p_1$.
Hence $p_1 - 1$ can have no prime divisors which implies $p_1 = 2$.
Now assume that for $j < i$
\[(p_j-1) = \prod_{j' = 1}^{j-1}p_{j'}^{a_{j'}}.\]
By iterating $f$ we will initially obtain divisors of $n$ of the form
\[\left(\prod_{j = 1}^{i-1} p_j^{b_j}\right)\left(\prod_{j = i}^{\ell} p_j^{a_j}\right) \]
where $0 \leq b_j \leq a_j$ for $1 \leq j < i$.

When we come to the divisor $p_i^{a_i}p_{i+1}^{a_{i+1}} \cdots p_{\ell}^{a_{\ell}}$, by Lemma~\ref{lem:minusone}
\[f(p_i^{a_i}p_{i+1}^{a_{i+1}} \cdots p_{\ell}^{a_{\ell}}) = (p_i -1)p_i^{a_i-1}p_{i+1}^{a_{i+1}} \cdots p_{\ell}^{a_{\ell}}.\]
Since $F_n = D_n$, $(p_i - 1)p_i^{a_i - 1}p_{i+1}^{a_{i+1}} \cdots p_{\ell}^{a_{\ell}}$ divides $n$ which in turn implies that $(p_i - 1)$ divides $p_1^{a_1}p_2^{a_2} \cdots p_{i-1}^{a_{i-1}}$.
If
\[(p_i-1) \neq \prod_{j = 1}^{i-1}p_j^{a_j} \]
since $f$ is a decreasing function, we see that
\begin{align*}
    f(p_i^{a_i}p_{i+1}^{a_{i+1}} \cdots p_{\ell}^{a_{\ell}}) &= (p_i -1)p_i^{a_i-1}p_{i+1}^{a_{i+1}} \cdots p_{\ell}^{a_{\ell}}\\
    &< p_1^{a_1}p_2^{a_2} \cdots p_{i-1}^{a_{i-1}}p_i^{a_i-1}p_{i+1}^{a_{i+1}} \cdots p_{\ell}^{a_{\ell}}
\end{align*}
and the divisor $p_1^{a_1}p_2^{a_2} \cdots p_{i-1}^{a_{i-1}}p_i^{a_i-1}p_{i+1}^{a_{i+1}} \cdots p_{\ell}^{a_{\ell}}$ will not be contained in $F_n$.
Thus when $F_n = D_n$ we must have
\[(p_i-1) = \prod_{j = 1}^{i-1}p_j^{a_j}\]
for all $1 \leq i \leq \ell$.
\end{proof}

\begin{algorithm}
\begin{algorithmic}
\State $n \gets 2$
\Loop 
\State $p \gets \text{largest prime divisor of $n$}$
\If{$n+1$ is prime} 
\State nondeterministically \textbf{choose} $n \gets np$ \textbf{or} $n \gets n(n+1)$
\Else 
\State $n \gets np$
\EndIf
\EndLoop
\end{algorithmic}
\caption{Nondeterministic algorithm to produce terms of sequence~\seqnum{A073935}.}
\label{alg:1}
\end{algorithm}

Algorithm~\ref{alg:1} is a nondeterministic algorithm which produces positive integers $n$ with $F_n = D_n$ (i.e., the terms of~\seqnum{A073935}).
Lemma~\ref{lem:A073935} says that if $n>1$ satisfies $F_n = D_n$, then always $F_{np} = D_{np}$ where $p$ is the largest prime divisor of $n$ and provided $n+1$ is prime $F_{n(n+1)} = D_{n(n+1)}$ also.
Notice in the case that $n+1$ is prime there are two ways to produce a new integer, and this causes the nondeterminism in Algorithm~\ref{alg:1}.
Terms of the sequence coming from various branches of the algorithm are shown in Figure~\ref{fig:Tree}.

\begin{theorem}
Every integer $n > 1$ with $F_n = D_n$ is a prime power pseudoperfect number.
\label{thm:pseudo}
\end{theorem}
\begin{proof}
Let $n > 1$ be such that $F_n = D_n$.
Assume $n$ has prime factorization
\[n = \prod_{i=1}^{\ell} p_i^{a_i}.\]
By Lemma~\ref{lem:A073935} we know that
\[(p_i-1) = \prod_{j = 1}^{i-1}p_j^{a_j}\]
for $1 \leq i \leq \ell$.
We define
\begin{align*}
    n_i &:= \frac{n}{p_i^{a_i}}\\
    n'_i &:= \frac{n}{\prod_{j=1}^i p_j^{a_j}}
\end{align*}
for $1 \leq i \leq \ell$ and define $n'_0 := n$.
We now compute
\begin{align*}
    \sum_{p^k \mid n} \frac{1}{p^k} + \frac{1}{n} &= \sum_{i=1}^{\ell} \sum_{j=1}^{a_i} \frac{1}{p_i^j} + \frac{1}{n}\\
    &= \sum_{i=1}^{\ell} \sum_{j=0}^{a_i-1} \frac{p_i^j n_i}{n} + \frac{1}{n}\\
    &= \sum_{i=1}^{\ell} \frac{(p_i^{a_i}-1)n_i}{(p_i -1)n} + \frac{1}{n}\\
    &= \sum_{i=1}^{\ell} \frac{(p_i^{a_i}-1)n'_i}{n} + \frac{1}{n}\\
    &= \sum_{i=1}^{\ell} \frac{p_i^{a_i}n'_i-n'_i}{n} + \frac{1}{n}\\
    &= \sum_{i=1}^{\ell} \frac{n'_{i-1} - n'_i}{n} + \frac{1}{n}\\
    &= \frac{n'_0 - n'_{\ell} + 1}{n}\\
    &= 1
\end{align*}
Therefore $n$ is prime power pseudoperfect.
\end{proof}

The converse of Theorem~\ref{thm:pseudo} is not true.
For example, the number $23994 = 2 \cdot 3^2 \cdot 31 \cdot 41$ is a prime power pseudoperfect number but $F_{23994} \neq D_{23994}$.

\begin{figure}
\centering
\begin{tikzpicture}
\Tree [.2 
        [.4 
            [.8 
                [.16 [.32 $\vdots$ ] [.272 $\vdots$ ] ]
            ]
            [.20 
                [.100 [.500 $\vdots$ ] [.10100 $\vdots$ ] ]
            ]
        ]
        [.6 
            [.18 
                [.54 [.162 $\vdots$ ] ]
                [.342 [.6496 $\vdots$ ] ]
            ]
            [.42 
                [.294 [.2058 $\vdots$ ] ]
                [.1806 [.77658 $\vdots$ ] ]
            ]
        ]
    ]
\end{tikzpicture}
\caption{Tree showing terms of sequence~\seqnum{A073935} on various branches of Algorithm~1.}
\label{fig:Tree}
\end{figure}
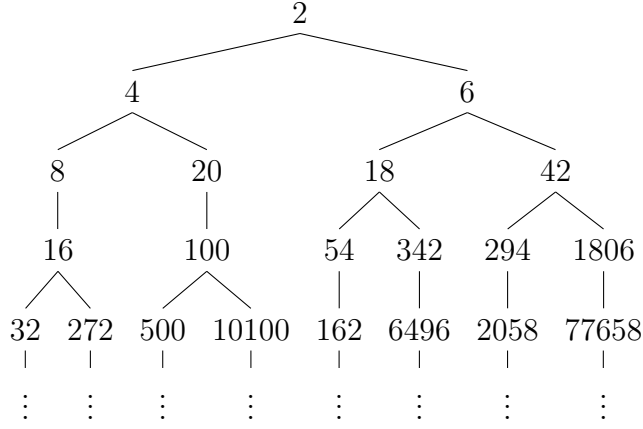

\section{Prime power Giuga numbers}
\label{sec:Giuga}

A \emph{prime power Giuga number} is a positive composite integer $n > 1$ which satisfies the Egyptian fraction condition
\[ \sum_{p^k \mid n} \frac{1}{p^k} - \frac{1}{n} \in \mathbb{N}\]
where the sum is taken over all prime power divisors of $n$.
Since Giuga numbers are squarefree it follows that all Giuga numbers are prime power Giuga numbers.
All prime power Giuga numbers we have found obey the stricter Egyptian fraction equation
\[ \sum_{p^k \mid n} \frac{1}{p^k} - \frac{1}{n} = 1 .\]
Prime power Giuga numbers are sequence~\seqnum{A286497}.

We now prove a lemma analogous to Lemma~\ref{lem:A073935}.

\begin{lemma}
Let $n > 1$ be a positive integer with prime factorization
\[n = \prod_{i=1}^{\ell} p_i^{a_i}\]
with $p_1 < p_2 < \cdots < p_{\ell}$ and $a_{\ell} = 1$.
If 
\[p_i-1 = \prod_{j = 1}^{i-1}p_j^{a_j}\]
for $1 \leq i < \ell$ and $p_{\ell} + 1 = \tfrac{n}{p_{\ell}}$,
then the positive integer $n$ is a prime power Giuga number.
\label{lem:Giuga}
\end{lemma}
\begin{proof}
Assume $n > 1$ is a positive integer satisfying the hypothesis of the lemma.
Then by Lemma~\ref{lem:A073935} and Theorem~\ref{thm:pseudo} we know the $\tfrac{n}{p_{\ell}}$ is a prime power pseudoperfect number.
So,
\begin{align*}
    \sum_{p^k \mid n} \frac{1}{p^k} - \frac{1}{n} 
    &= \sum_{p^k \mid \tfrac{n}{p_{\ell}}} \frac{1}{p^k} + \frac{1}{p_{\ell}} - \frac{1}{n}\\
    &= \frac{\frac{n}{p_{\ell}}-1}{\tfrac{n}{p_{\ell}}} + \frac{1}{p_{\ell}} - \frac{1}{n}\\
    &= \frac{p_{\ell}}{p_{\ell}+1} + \frac{1}{p_{\ell}} - \frac{1}{(p_{\ell} + 1)p_{\ell}}\\
    &= 1.
\end{align*}
\end{proof}

Lemma~\ref{lem:Giuga} gives a sufficient but not necessary condition for being a prime power Giuga number.
Table~\ref{tbl:Giuga} shows prime power Giuga numbers less than $10^7$.
Notice some numbers in the table, such as $858 = 2\cdot 3 \cdot 11 \cdot 13$, do not satisfy the condition in Lemma~\ref{lem:Giuga}.

\begin{table}
    \centering
    \[
    \begin{array}{|c|c|}
    \hline
    n & \text{prime factorization} \\ \hline
    12  &  2^2 \cdot 3  \\ \hline
30  &  2 \cdot 3 \cdot 5  \\ \hline
56  &  2^3 \cdot 7  \\ \hline
306  &  2 \cdot 3^2 \cdot 17  \\ \hline
380  &  2^2 \cdot 5 \cdot 19  \\ \hline
858  &  2 \cdot 3 \cdot 11 \cdot 13  \\ \hline
992  &  2^5 \cdot 31  \\ \hline
1722  &  2 \cdot 3 \cdot 7 \cdot 41  \\ \hline
2552  &  2^3 \cdot 11 \cdot 29  \\ \hline
2862  &  2 \cdot 3^3 \cdot 53  \\ \hline
16256  &  2^7 \cdot 127  \\ \hline
30704  &  2^4 \cdot 19 \cdot 101  \\ \hline
66198  &  2 \cdot 3 \cdot 11 \cdot 17 \cdot 59  \\ \hline
73712  &  2^4 \cdot 17 \cdot 271  \\ \hline
86142  &  2 \cdot 3 \cdot 7^2 \cdot 293  \\ \hline
249500  &  2^2 \cdot 5^3 \cdot 499  \\ \hline
629802  &  2 \cdot 3^3 \cdot 107 \cdot 109  \\ \hline
1703872  &  2^6 \cdot 79 \cdot 337  \\ \hline
6127552  &  2^6 \cdot 67 \cdot 1429  \\ \hline
    \end{array}
    \]
    \caption{Prime power Giuga numbers less than $10^7$.}
    \label{tbl:Giuga}
\end{table}

\section{Producing more terms}
\label{sec:more}

In this section we give some formulas that can be helpful in finding solutions to our Egyptian fraction equations.
Similar results for primary pseudoperfect numbers and Giuga numbers are given in the article~\cite[Theorem 8]{ST14}.
The article~\cite[Proposition 1]{SM17} also contains conditions for primary pseudoperfect numbers.
Results to help search for solutions of other related Egyptian fraction equations can be found in the articles~\cite[Proposition 12, Lemma 17]{BH88} and~\cite[Lemma 4.1, Lemma 4.2]{PPN}.

\begin{proposition}
Let $n > 1$ be a positive integer.
\begin{enumerate}
    \item[(i)] If $n$ is in the sequence~\seqnum{A073935} and $p$ is largest prime divisor of $n$, then both $\tfrac{n}{p}$ and $np$ are in the sequence~\seqnum{A073935}.
    \item[(ii)] If $n$ is in the sequence~\seqnum{A073935} and $n+1$ is prime, then $n(n+1)^k$ is in the sequence~\seqnum{A073935} for any nonnegative integer $k$.
    \item[(iii)] If $n$ is a prime power pseudoperfect number and $n+1$ is prime, then $n(n+1)^k$ is a prime power pseudoperfect number for any nonnegative integer $k$.
    \item[(iv)] If $n$ is a prime power pseudoperfect number and $n-1$ is prime, then $n(n-1)$ is a prime power Giuga number.
\end{enumerate}
\label{prop:new}
\end{proposition}
\begin{proof}
Parts (i) and (ii) follow immediately from Lemma~\ref{lem:A073935}.

For part (iii) assume that $n$ is a prime power pseudoperfect number and $n+1$ is prime.
So,
\begin{align*}
    \sum_{p^k \mid n(n+1)^k} \frac{1}{p^k} + \frac{1}{n(n+1)^k} &= \sum_{p^k \mid n} \frac{1}{p^k}  + \sum_{j=1}^k \frac{1}{(n+1)^j} + \frac{1}{n(n+1)^k}\\
    &=  \frac{n-1}{n} + \frac{(n+1)^k-1}{n(n+1)^k} + \frac{1}{n(n+1)^k}\\
    &= \frac{(n-1)(n+1)^k + (n+1)^k}{n(n+1)^k}\\
    &= 1.
\end{align*}

For part (iv) assume that $n$ is a prime power pseudoperfect number and $n-1$ is prime.
So, 
\begin{align*}
    \sum_{p^k \mid n(n+1)} \frac{1}{p^k} - \frac{1}{n(n-1)} &= \sum_{p^k \mid n} \frac{1}{p^k}  + \frac{1}{(n+1)} - \frac{1}{n(n-1)}\\
    &=  \frac{n-1}{n} + \frac{1}{(n+1)} - \frac{1}{n(n-1)}\\
    &= \frac{(n-1)(n-1) +n -1}{n(n-1)}\\
    &= 1.
\end{align*}
\end{proof}

Consider the number
\[n = 23994 = 2 \cdot 3^2 \cdot 31 \cdot 43\]
which is a prime power pseudoperfect number.
However, neither
\[\frac{n}{43} = 558 = 2 \cdot 3^2 \cdot 31  \]
nor
\[43n = 1031742 = 2 \cdot 3^2 \cdot 31 \cdot 43^2 \]
is a prime power pseudoperfect number.
Hence, a version of Proposition~\ref{prop:new} (i) does not hold from prime power pseudoperfect numbers.
Also consider the number $n = 18$ which is a prime power pseudoperfect number, and the number $n(n-1) = 306$ is a prime power Giuga number since $n-1 = 17$ is prime.
However, the number $n(n-1)^2 = 5202$ is not a prime power Giuga number.
Thus a version of Proposition~\ref{prop:new} (ii) or (iii) does not hold for prime power Giuga numbers.

\section{Open questions}
\label{sec:open}

Proposition~\ref{prop:new} immediately shows that there are infinitely many terms in both the sequence~\seqnum{A073935} and the sequence of prime power pseudoperfect numbers~\seqnum{A283423}.
Proposition~\ref{prop:new} does not give a way to produce infinitely many prime power Giuga numbers, but we conjecture there are infinitely many such numbers.

\begin{conjecture}
There are infinitely many prime power Giuga numbers.
\label{conj:G}
\end{conjecture}

A \emph{Mersenne prime} is prime number $p$ such that $p= 2^k-1$ for some integer $k$.
Mersenne primes are sequence~\seqnum{A000668}.
By Lemma~\ref{lem:Giuga}, the number $n = 2^k (2^k - 1)$ is a prime power Giuga number whenever $2^k-1$ is a Mersenne prime.
Hence, Conjecture~\ref{conj:G} would follow from an infinitude of Mersenne primes, and it is believed that there are infinitely many Mersenne primes.

A \emph{Fermat prime} is prime number $p$ such that $p = 2^k+1$ for some positive integer $k$.
Fermat primes are sequence~\seqnum{A019434}.
By Lemma~\ref{lem:A073935} the number $2^k$ is in the sequence~\seqnum{A073935} for any positive integer $k$, and the powers of $2$ are the only numbers in the sequence that have a unique prime divisor.
If a number with two distinct prime divisors is in sequence~\seqnum{A073935} it must be of the form  $2^k(2^k + 1)^j$ where $2^k + 1$ is a Fermat prime and $j$ is a positive integer.

\begin{table}
    \centering
    \[
    \begin{array}{|c|c|c|}
    \hline
    p & p-1 & \text{level} \\ \hline
    2 & 1 & 0 \\ \hline
3  &  2 & 1 \\ \hline
5  &  2^2 & 1 \\ \hline
7  &  2 \cdot 3 & 2\\ \hline
17  &  2^4 & 1 \\ \hline
19  &  2 \cdot 3^2 & 2 \\ \hline
43  &  2 \cdot 3 \cdot 7 & 3 \\ \hline
101  &  2^2 \cdot 5^2 & 2\\ \hline
163  &  2 \cdot 3^4 & 2\\ \hline
257  &  2^8 & 1\\ \hline
487  &  2 \cdot 3^5 & 2\\ \hline
1459  &  2 \cdot 3^6 & 2\\ \hline
14407  &  2 \cdot 3 \cdot 7^4 & 3 \\ \hline
26407  &  2 \cdot 3^4 \cdot 163 & 3\\ \hline
39367  &  2 \cdot 3^9 & 2\\ \hline
62501  &  2^2 \cdot 5^6 & 2\\ \hline
65537  &  2^{16} & 1\\ \hline
77659  &  2 \cdot 3 \cdot 7 \cdot 43^2 & 4\\ \hline
1020101  &  2^2 \cdot 5^2 \cdot 101^2 & 3\\ \hline
    \end{array}
    \]
    \caption{Table of extended Fermat primes $p$ along with factorizations of $p-1$.}
    \label{tbl:primes}
\end{table}

The primes which occur as divisors of terms of the sequence~\seqnum{A073935} are primes $p$ such that
\[ p-1 = \prod_{i=1}^{\ell} p_i^{a_i}\]
where for $1 \leq i \leq \ell$
\[p_i - 1 = \prod_{j=1}^{i-1}p_j^{a_i}.\]
If $p$ is a prime such that $p-1 = \prod_{i=1}^{\ell}p_i^{a_i}$ we say that $p$ is a \emph{level-$\ell$ extended Fermat prime}.
A table of such primes is included in Table~\ref{tbl:primes}.
By convention the prime $2$ is the only level-$0$ extended Fermat prime.
With this new definition usual Fermat primes are now level-$1$ extended Fermat primes.
Extended Fermat primes are sequence~\seqnum{A286499}.
It is thought that there are only finitely many Fermat primes.
However, we believe there are infinitely many extended Fermat primes and offer the following conjectures.

\begin{figure}
\centering
\begin{tikzpicture}
\Tree [.2 
        [.3 [.7 [.43 [.77659 $p$ ] $\cdots$ ] [.14407 $\vdots$ ] $\cdots$ ] [.19 $\vdots$ ] [.163 $\vdots$ ] $\cdots$ ]
        [.5 [.101 $\vdots$ ] [.62501 $\vdots$ ] $\cdots$ ]
        $\cdots$
    ]
\end{tikzpicture}
\caption{A portion of the tree of extended Fermat primes.}
\label{fig:FermatTree}
\end{figure}
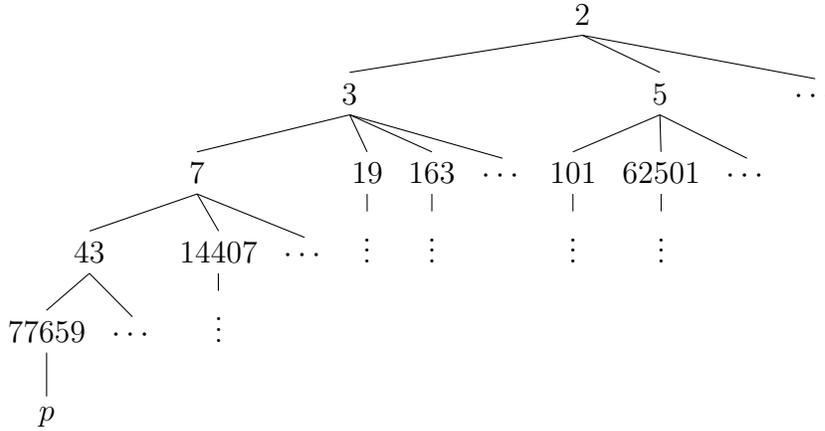

\begin{conjecture}
There exists an extended Fermat prime $p$ such that $(p-1)p^k + 1$ is an extended Fermat prime for infinitely many values of $k$.
\label{conj:B}
\end{conjecture}

\begin{conjecture}
Given any positive integer $\ell$ there exists a level-$\ell$ extended Fermat prime.
\label{conj:C}
\end{conjecture}

Towards an answer to Conjecture~\ref{conj:B}, the prime $3$ may give a example.
Computation suggests there are many primes of the form $2\cdot 3^k + 1$.
The values of $k$ for which $2\cdot 3^k + 1$ is prime is sequence~\seqnum{A003306}.
In the direction of Conjecture~\ref{conj:C}, we have found a level-$5$ extended Fermat prime
\[p  =  2 \cdot 3 \cdot 7 \cdot 43^2 \cdot 77659^{197} + 1.\]
We can form a rooted tree of extended Fermat primes with root $2$ as follows.
Let $p_1$ and $p_2$ be two extended Fermat primes, then $p_2$ is a descendant of $p_1$ if and only if $p_1 \mid (p_2-1)$.
A portion of this tree, including a path to the level-$5$ extended Fermat prime $p$, is shown in Figure~\ref{fig:FermatTree}.

\section{Acknowledgments}

The author wishes to thank the anonymous referee and Jeffrey O. Shallit for their helpful comments which have improved this paper.

\bigskip
\hrule
\bigskip

\noindent 2010 {\it Mathematics Subject Classification}: Primary 11D68;
Secondary 11A41, 11A51.

\noindent \emph{Keywords:}
Egyptian fraction, pseudoperfect number, primary pseudoperfect number, Giuga number.

\bigskip
\hrule
\bigskip

\noindent (Concerned with sequences
\seqnum{A000668}, 
\seqnum{A003306}, 
\seqnum{A005835}, 
\seqnum{A007850}, 
\seqnum{A019434}, 
\seqnum{A054377}, 
\seqnum{A073932}, 
\seqnum{A073935}, 
\seqnum{A283423},
\seqnum{A286497}, and
\seqnum{A286499}.) 
\bigskip
\hrule




\end{document}